\documentclass[12pt]{article}
\usepackage[cp1251]{inputenc}

\usepackage[T2A]{fontenc}

\usepackage[english]{babel}
\usepackage{mathtext}
\usepackage[dvips]{graphicx}
\usepackage{amsmath}
\usepackage{amssymb}
\usepackage{amsxtra}
\usepackage{latexsym}
\usepackage{ifthen}

\def\pr{\mathsf P}

\begin{document}

\newcounter{lemma}
\newcommand{\lem}{\par \refstepcounter{lemma}
{\bf Lemma \arabic{lemma}.}}
\renewcommand{\thelemma}{\arabic{lemma}}

\newcounter{theorem}[section]
\newcommand{\thm}{\par \refstepcounter{theorem}
{\bf Theorem \arabic{theorem}.}}
\renewcommand{\thetheorem}{\arabic{theorem}}

\newcounter{example}
\newcommand{\example}{\par \refstepcounter{example}
{\bf Example \arabic{example}.}}
\renewcommand{\theexample}{\arabic{example}}

\newcommand{\proof}{{\it Proof.\,\,}}

\title{Peano phenomenon for\\ stochastic equations with local time.}
\author{Ivan H. Krykun\thanks{Department of Probability, Institute of Applied Mathematics and Mechanics,
National Academy of Science of Ukraine, Donetsk, Ukraine. \emph{e-mail :} iwanko@i.ua}}
\date{}

\maketitle

\begin{abstract}
In this paper we investigate weak convergence of measures generated
by solutions of stochastic equations with local time and small
diffusion while the last one tends to zero. In case the
correspondent ordinary differential equation has infinitely many
solutions we prove that limit measure concentrated with some weights
on its extreme solutions. Formulae for weights are obtained.


\end{abstract}

\section{Introduction.}

The problem of convergence of measures generated by solutions of
stochastic differential equations (SDEs) with small diffusion
$$x_\varepsilon(t)=x_0+\int_0^tb(x_\varepsilon(s))ds+\varepsilon w(t),$$
as $\varepsilon\rightarrow0$ to the measure concentrated on the
unique solution of the ordinary differential equation (ODE)
$$x'(t)=b(x(t)), x(0)=x_0,$$
was considered in several papers. Mention books by Stroock-Varadhan \cite{bookstrvar} and Wentzell-Freidlin
\cite{bookvenfr}.
Non-uniqueness case of solutions of ODE (Peano phenomenon) also was
considered. Mention such authors as Baldi \cite{paperbaldifr},
Baldi-Bafico \cite{paperbaldi}, Veretennikov \cite{paperver83},
Gradinaru-Herrmann-Roynette \cite{papergrad},
Buckdahn-Quincampoix-Ouknine \cite{paperbuck}, Krykun-Makhno
\cite{krmah13}.

In this paper we consider SDE with local time, continuous drift and small diffusion
\begin{equation}\label{sdrlch}
\xi_\varepsilon(t)=\beta L^{\xi_\varepsilon}(t,0)+ \int_0^t
b(\xi_\varepsilon(s))ds+\varepsilon\int_0^t
\sigma(\xi_\varepsilon(s))dw(s), \qquad t\in [0,1].
\end{equation}

With this equation we connect Cauchy problem for ODE
\begin{equation}\label{zdr}
y'(t)=b(y(t)),\hskip 10pt y(0)=0.
\end{equation}
and we prove weak convergence of measures generated by solutions of
(\ref{sdrlch}) to measure concentrated with some weights on extreme
solutions of (\ref{zdr}). Formulae for calculating their weights are
obtained.

We use the following notation: $I_A(x)$ be the indicator function of the
set $A$; $a^+=\max(a,0);$ $\mathbb{C}[0,\infty) -$ be a space of continuous functions $f(t),$ $t\in [0,\infty),$ with the metric of uniform convergence on compact subsets of $[0,\infty)$:
$$\rho (f,g)=\sum^\infty_{N=1}\frac{1}{2^N}\frac{\sup_{t\in [0,N]}|f(t)-g(t)|}{1+\sup_{t\in [0,N]}|f(t)-g(t)|}.$$
Let $\mathfrak B$ be $\sigma -$ algebra of Borel sets of this space.
Denote $f(x)\sim g(x)$ as $x\rightarrow x_0$, if equality
$$\lim_{x\rightarrow x_0}\frac{f(x)}{g(x)}=1$$
takes place.

Let $(\Omega, \mathfrak{F}, \mathfrak{F}_{t}, \pr)$ be the probability space with a flow
of $\sigma$-algebras $ \mathfrak{F}_{t}$,   $t\geq0$,  $(w(t), \mathfrak{F}_{t} )$ be standard one-dimensional Wiener process.

The function $\text{sgn} (x)$ is defined as:
$$
\text{sgn}( x)=\left \{ \aligned
1, \hskip 5 pt & \text{if} \hskip 5 pt x>0,\\
0, \hskip 5 pt & \text{if} \hskip 5 pt x=0,\\
-1, \hskip 5 pt & \text{if} \hskip 5 pt x<0.
\endaligned
\right.
$$

\textbf{Definition 1.} \cite[Definition 4.7(1)]{paperengsch} Equation (\ref{sdrlch}) has a \textbf{weak solution} if for given functions $b(x),$ $\sigma (x)$ and a constant $\beta$ there exist a probability space $(\Omega, \mathfrak{F}, \mathfrak{F}_{t}, \pr)$  with flow of $\sigma$-algebras $\mathfrak{F}_{t},$ $t\geq0$, continuous semimartingale $(\xi(t), \mathfrak{F}_{t})$ and a standard one-dimensional Wiener process $(w(t), \mathfrak{F}_{t} )$ such as
\begin{equation}\label{l}
L^\xi(t,0)=\lim_{\delta\rightarrow0}\frac{1}{2\delta}\int_0^tI_
{(-\delta,\delta)}(\xi(s))ds
\end{equation}
exists almost surely and (\ref{sdrlch}) takes place almost surely.

\textbf{Definition 2.} \cite[Definition 4.7(2)]{paperengsch} Equation (\ref{sdrlch}) has a \textbf{strong solution} if for the given functions $b(x),$ $\sigma (x)$ and a constant $\beta$ relations (\ref{sdrlch}) and (\ref{l}) take place almost surely (a.s.) on the given probability space $(\Omega, \mathfrak{F}, \mathfrak{F}_{t}, \pr)$  with flow of $\sigma$-algebras $\mathfrak{F}_{t},$ $t\geq0$ and given Wiener process $(w(t), \mathfrak{F}_{t} ).$

\textbf{Definition 3.} \cite[Chapter IV, Definition 1.4]{VatIke81}
 The \textbf{uniqueness of solutions} (or \textbf{weak uniqueness} or \textbf{uniqueness in the sense of probability}) for equation (\ref{sdrlch}) holds if whenever $X$ and $X'$ are two solutions of (\ref{sdrlch}) such as $X(0)=0$ a.s. and $X'(0)=0$ a.s., then the laws on the space $\mathbb{C}([0, \infty])$ of the processes $X$ and $X'$ coincide.

\textbf{Definition 4.} \cite[Chapter IV, Definition 1.5]{VatIke81}
We say that the \textbf{pathwise uniqueness} (or \textbf{strong uniqueness}) \textbf{of solutions} for equation (\ref{sdrlch}) holds if whenever $X$ and $X'$ are any two solutions defined on the same probability space $(\Omega, \mathfrak{F}, \mathfrak{F}_{t}, \pr)$ with the same flow of $\sigma$-algebras $\mathfrak{F}_{t},$ and the same one-dimensional Wiener process $(w(t), \mathfrak{F}_{t})$ such as $X(0) = X'(0)$ a.s., then $X(t) =X'(t)$ for all $t>0$ a.s.

\vskip 5 pt
\noindent For coefficients of equation (\ref{sdrlch}) we consider next condition $( I )$.
\vskip 5 pt

\textbf{Condition (I)}: \\
$I_1.$ Function $b(x)$ be continuous, $b(0)=0$ and null point is
its unique zero. \\
$I_2.$ There is a constant $\Lambda\geq1$ such as
$$ |b(x)|^2 + \sigma^2(x)\leq\Lambda(1+|x|^2), \, \, \sigma^2(x)\geq\Lambda^{-1}.$$
$I_3.$ For every $x$ and $y$ the product $\sigma(x)\sigma(y)>0$ and for
every $N<\infty$
$$\sup_{-N=x_0<...<x_k=N}\sum^k_{i=1} |\sigma (x_i)-\sigma
(x_{i-1})|<\infty.$$
$I_4.$ The constant $|\beta|<1.$

\vskip 5 pt

The paper is organized as follows. Section 2 contains results for
ordinary differential equations. Main results of the paper are
formulated in Section 3 and are proved in Section 4. Section 5
contains some examples. \vskip 10 pt

\section{Ordinary Differential Equations.}

Consider Cauchy problem (\ref{zdr}). By conditions $I_1$ and $I_2$
for the function $b(x)$ this problem has at least one solution --
zero -- and all the solutions of this problem pass through the point
$(0; 0)$. This set of all curves is called integral funnel and denote by
$\mathfrak R$. From existence of two different solutions it follows
that there are infinitely many solutions. Each solution from the
integral funnel is located between two special solutions which are
called extremal - upper $\overline{y}(t)$ and lower
$\underline{y}(t)$, where $\overline{y}(t)=\sup\{ y(t), y(t)\in
\mathfrak R\}$, $\underline{y}(t)=\inf\{y(t), y(t)\in \mathfrak
R\}$.

\vskip 5 pt
Note that if $b(x)x<0 $ for $x\neq0,$ then the problem (\ref{zdr}) has only zero solution.

\vskip 5 pt
Existence of nonzero solution of (\ref{zdr}) requires convergence at least one of integrals
\begin{equation}\label{cond1}
\int^\delta_0\frac{1}{b(y)}dy, \,\,\int^0_{-\delta}\frac{1}{b(y)}dy.
\end{equation}

So interesting are the following cases:
\vskip 5 pt

\noindent$A_1.$ Function $b(x)x>0$ for $x\neq0$ and both integrals in (\ref{cond1}) are convergent.\\
$A_2.$ Function $b(x)x>0$ for $x\neq0$ and the first integral in (\ref{cond1}) converges and the second one diverges.\\
$A_3.$ Function $b(x)x>0$ for $x\neq0$ and the first integral in (\ref{cond1}) diverges and the second one converges.\\
$A_4.$ Function $b(x)>0$ for $x\neq0$ and the first integral in (\ref{cond1}) converges.\\
$A_5.$ Function $b(x)<0$ for $x\neq0$ and the second integral in (\ref{cond1}) converges.

\vskip 5 pt
Let $\displaystyle H(x)=\int_0^x\frac{1}{b(y)}dy$ for $x\geq0$ and $\displaystyle K(x)=\int_x^0\frac{1}{b(y)}dy$ for $ x\leq0$. If condition $I_1$ takes place, then these functions are strictly monotone. Denote the inverse functions of $H(x)$ and $K(x)$ by $H^{-1}(x)$ and $K^{-1}(x)$ respectively.

\vskip 5 pt

\begin{lem}\label{lem1} \cite[Lemma 2.2 and Lemma 2.3]{krmah13}.\\
{\it \textbf{1.} In case $A_1$
all nonzero solutions of the problem (\ref{zdr}) are given by:
\begin{equation}\label{ylam}
y_\lambda(t)=H^{-1}\big((t-\lambda)^+\big),\;\;\lambda\geq0,
\end{equation}
\begin{equation}\label{ymu}
y_\mu(t)=K^{-1}\big(-(t-\mu)^+\big),\;\;\mu\geq0.
\end{equation}
In this case the extremal solutions are $\overline{y}(t)=H^{-1}(t), $ $\underline{y}(t)=K^{-1}(t).$ \\
\textbf{2.} In cases $ A_2 $ and $ A_4 $
all nonzero solutions of the problem (\ref{zdr}) have the form (\ref{ylam}) and the
 extremal solutions are $\overline{y}(t)=H^{-1}(t),$ $\underline{y}(t)=0.$\\
\textbf{3.} In cases $ A_3 $ and $ A_5 $
all nonzero solutions of the problem (\ref{zdr}) have the form (\ref{ymu}) and the
extremal solutions are $\overline{y}(t)=0,$ $\underline{y}(t)=K^{-1}(t).$}
\end{lem}

\vskip 5 pt For study of the weights of a limit measure we need to compute an expression
\begin{equation}\label{gamma}\Gamma_K=
\lim_{\varepsilon\rightarrow 0}\frac{-A_\beta^\varepsilon(-K)}{A_\beta^\varepsilon(K)-A_\beta^\varepsilon(-K)},
\end{equation}
where \hskip 15 pt
$\displaystyle A_\beta^\varepsilon(x)=\int_0^x\exp\Biggl\{-\frac{2}{\varepsilon^2}\int_0^z \frac{(1+\beta \text{sgn} (v))b
\big((1+\beta \text{sgn} (v))v\big)}{\sigma^2\big((1+\beta \text{sgn} (v))v\big)}dv\Biggl\}dz.$

\vskip 5 pt
To calculate $\Gamma_K $ put
\begin{equation}\label{lbig}
L(x)=\int^x_0 \frac{b(y)}{\sigma^2(y)}dy.
\end{equation}

\begin{lem}\label{glem2} {\it Let $b(x)x>0$ for $x\neq 0,$
for some constants $d >0,$ $\delta>0$ and $\gamma$ as $x\to 0+$}
\begin{equation}\label{glem11}
L(x)|\ln L(x)|^\gamma\sim dx^\delta
\end{equation}
{\it and for some constants $k>0,$  $\mu>0$ and $\theta$ as $x\to 0-$}
\begin{equation}\label{glem21}
L(x)|\ln L(x)|^\theta\sim k|x|^\mu.
\end{equation}
{\it Then the value of $\Gamma_K$ is independent on $K$ (so we denote it by $\Gamma$) and take place the following statements:\\
1. If $\delta=\mu$ and $\gamma=\theta$, then
$$\Gamma= \frac{1}{1+\frac{1-\beta}{1+\beta}\Big(\frac{k}{d}\Big)^{\frac{1}{\delta}}}.$$
2. If $\delta<\mu$ or $\delta=\mu\;$ and  \; $\gamma<\theta\;$, then $\Gamma=1$.\\
3. If $\delta>\mu$ or $\delta=\mu\;$ and  \; $\gamma>\theta\;$, then $\Gamma=0$.}
\end{lem}
\vskip 5 pt
 {\it Proof.}
Denote
$$L_\beta(x)=\int^x_0 \frac{(1+\beta \text{sgn} (y))b\big((1+\beta \text{sgn} (y))y\big)}{\sigma^2\big((1+\beta \text{sgn} (y))y\big)}dy.
$$
Let $x>0$ then:
$$L_\beta(x)=\int_0^x\frac{(1+\beta) b\big((1+\beta)y\big)}{a\big((1+\beta)y\big)}dy=
\int_0^{(1+\beta)x}\frac{b\big(y\big)}{a\big(y\big)}dy=L\big((1+\beta)x\big),$$
where function $L(x)$ defined by (\ref{lbig}).
So from the condition (\ref{glem11}) we have
$$L_\beta(x)|\ln L_\beta(x)|^\gamma\sim d(1+\beta)^\delta x^\delta=d^*x^\delta,$$
where $d^*=d(1+\beta)^\delta$.

For $x<0$ from the condition (\ref{glem21}) by analogy we have
$$L_\beta(x)|\ln L_\beta(x)|^\gamma\sim k^*|x|^\mu,$$
where $k^*=k(1-\beta)^\mu$.

Using now \cite[lemma 2.8]{krmah13} we get a statement of the lemma \ref{glem2}.\\

 {\it Lemma \ref{glem2} is proved.}

\vskip 10 pt

\section {Main results.}
\vskip 5 pt It is known that if conditions $ I_2 $ and $I_4$ are
hold, then there exists a weak unique weak solution of equation
(\ref{sdrlch}) \cite[theorem 4.35]{paperengsch}. By connection
between stochastic equations with local time and Ito stochastic
equations \cite{papermahno01}, one can prove next theorem.

\vskip 5 pt
\begin{thm}\label{exst} {\it Let for the coefficients of the equation (\ref{sdrlch}) conditions $I_2$, $I_3$ and $I_4$ are hold. Then for every fixed $\varepsilon>0$ equation (\ref{sdrlch}) has a pathwise unique strong solution.}
\end{thm}
\vskip 5 pt

Consider the stochastic equation with local time (\ref{sdrlch}) and corresponding to it Cauchy problem (\ref{zdr}). Denote by $\mu_\varepsilon$ measures generated by processes $\xi_\varepsilon(\cdot)$ on the space $(\mathbb{C}[0,\infty), \mathfrak B) .$

\vskip 5 pt
\begin{thm}\label{thbaldi}
 {\it Suppose that the conditions $I_1,$ $I_2,$ $І_4,$ $A_1,$ (\ref{glem11}), (\ref{glem21}) are hold for
the coefficients of equation (\ref{sdrlch}). Then for measures $\{\mu_\varepsilon\}$ and for any bounded
continuous functional $F$ defined on the space $\mathbb{C}[0,\infty)$, the equality
\begin{equation}\label{osn}
\lim_{\varepsilon\to 0}\int_{\mathbb{C}[0,\infty)} F(f)\mu_\varepsilon(df)= \Gamma F(\overline{y})+
\big(1-\Gamma\big)F(\underline{y}),
\end{equation}
takes place, where $\overline{y}, \underline{y}$ are extremal solutions of the problem (\ref{zdr}),
and the value of $\Gamma$ defined by the lemma \ref{glem2}.}
\end{thm}
\vskip 5 pt

For the investigation of cases $ A_2 - A_5 $ we will use a comparison
theorem. So we need a pathwise unique strong solutions of SDEs.

\vskip 5 pt

\begin{thm}\label{2}
 {\it Suppose that for the coefficients of the equation (\ref{sdrlch}) the condition (I) holds.

In the cases $A_2$, $A_4$ and if condition (\ref{glem11}) takes place, then
limit measure for the sequence $\{\mu_\varepsilon\}$ is concentrated on the upper extremal solution of the problem (\ref{zdr}).

In the cases $A_3$, $A_5$ and if condition (\ref{glem21}) takes place, then
limit measure for the sequ\-ence $\{\mu_\varepsilon\}$ is con\-cent\-rated
on the lower extremal solution of the problem (\ref{zdr}).}
\end{thm}

\vskip 10 pt

\section {Proof of theorems.}

{\it Proof of theorem \ref{exst}}.

Solution of equation (\ref{sdrlch}) is strongly connected with solution of Ito's equation. Denote
\begin{equation}\label{kappa}
\kappa(x)=\left\{\aligned
(1-\beta)x, \quad & x\leq0 , \\ (1+\beta)x,\quad & x\geq0 ,
\endaligned \right.
\end{equation}
and let $\;\;\;\varphi(x)=\left\{\aligned
\frac{x}{1-\beta}, \quad & x\leq0 \\ \frac{x}{1+\beta},\quad & x\geq0
\endaligned \right.$\;\;\; be inverse function of $\kappa(x)$.
Further set
$$\widetilde{b}(x)=\frac{b(\kappa(x))}{1+\beta \text{sgn} (x)},\hskip 15 pt \widetilde{\sigma}(x)=\frac{\sigma(\kappa(x))}{1+\beta \text{sgn} (x)},$$
and consider Ito's stochastic equation:
\begin{equation}\label{sdrito}
\eta_\varepsilon(t)=\int_0^t \widetilde{b}(\eta_\varepsilon(s))ds+\varepsilon\int_0^t\widetilde{\sigma}(\eta_\varepsilon(s))dw(s), \qquad t\in [0,1].
\end{equation}

For function $\widetilde{b}(t)$, $\widetilde{\sigma}(t)$ also take
place conditions $I_2$ and $I_3$ as for functions $b(t)$, $\sigma(t)$.
Consequently we can use \cite[theorem 3.2]{krmah13} for equation (\ref{sdrito}). So for every fixed $\varepsilon>0$ equation (\ref{sdrito}) has a pathwise unique strong solution. But from \cite[Lemma 1]{papermahno03} we have
$\eta_\varepsilon(t)=\varphi(\xi_\varepsilon(t))$ or $\xi_\varepsilon(t)=\kappa(\eta_\varepsilon(t))$ so for every fixed $\varepsilon>0$ equation (\ref{sdrlch}) also has a pathwise unique strong solution.\\

{\it Theorem \ref{exst} is proved}.

\vskip 5 pt
Equation (\ref{sdrito}) corresponds with next Cauchy problem:
\begin{equation}\label{zdr2}
z'(t)=\widetilde{b}(z(t)),\hskip 10pt z(0)=0.
\end{equation}

By the lemma \ref{lem1} for the equation (\ref{zdr2}) we have next result: every solution of this problem is one of two following types
$$z_\lambda(t)=\widetilde{H}^{-1}\big((t-\lambda)^+\big),\;\;\lambda\geq0,$$
$$z_\mu(t)=\widetilde{K}^{-1}\big(-(t-\mu)^+\big),\;\;\mu\geq0,$$
where $\displaystyle \widetilde{H}(x)=\int_0^x\frac{1}{\widetilde{b}(y)}dy$ for $x\geq0$ ; $\displaystyle\widetilde{K}(x)=\int_x^0\frac{1}{\widetilde{b}(y)}dy$ for $x\leq0$.

It is clear that $\widetilde{H}^{-1}(t)$ and
$\widetilde{K}^{-1}(-t)$ are the solutions which first leave
$[-r,r]$. This functions are called extremal (upper and lower
respectively) and denote by $\overline{z}(t)$, $\underline{z}(t)$.

\vskip 5 pt

It is clear that  strictly positive and strictly increasing  for
$t>0$ functions $\overline{y}(t)$, $\overline{z}(t)$ are solutions
of problem (\ref{zdr}) and problem (\ref{zdr2}) respectively.
Similarly, strictly negative and strictly decreasing  for $t>0$
functions $\underline{y}(t)$, $\underline{z}(t)$ are solutions of
problem (\ref{zdr}) and problem (\ref{zdr2}) respectively. Let's
prove connection between extremal solutions of problem (\ref{zdr})
and problem (\ref{zdr2}).
\vskip 5 pt

\begin{lem}\label{lem2}
$\overline{y}(t)=\kappa(\overline{z}(t))$, $\underline{y}(t)=\kappa(\underline{z}(t))$.
\end{lem}

\begin{proof} Let's prove lemma for the function $\overline{y}(t)$, for the function $\underline{y}(t)$ it be analogously.
Let the function $y(t)$ be any solution of equation (\ref{zdr}).
Consider the function $z(t)=\varphi(y(t)).$ Functions from the set
$\mathfrak R$ don't change their signs, so we get from equation (\ref{zdr})
$$
\begin{aligned}
z(t)&= \frac{y(t)}{1+\beta\text{sgn}
y(t)}=\frac{1}{1+\beta\text{sgn} y(t)}\int^t_0
b(\kappa(\varphi(y(s)))ds=\\
&=\int^t_0\frac{ b(\kappa(\varphi(y(s)))}{ 1+\beta\text{sgn}
\varphi(y(s)) }ds= \int^t_0 \tilde b(z(s))ds.
\end{aligned}
$$
Thus the function $z(t)=\varphi(y(t))$ is a solution of equation
(\ref{zdr2}). The function $\varphi(x)$ is strictly increasing and
we have
$$
z(t)=\varphi(y(t))\leq\varphi(\overline{y}(t))=\overline{z}(t).
$$

{\it Lemma \ref{lem2} is proved.}
\end{proof}

Denote by $\nu_\varepsilon$ measures generated by the processes
$\eta_\varepsilon(\cdot)$ on the space $(C[0,\infty), \mathfrak B).$

{\it Proof of theorem \ref{thbaldi}}. From conditions of the
theorem \ref{thbaldi} it implies that the coefficients of the
process $\eta_\varepsilon(t)$ satisfy conditions \cite[theorem
4.1]{krmah13}. Therefore the measures $\{\nu_\varepsilon\}$ weakly
converge to the measure $\nu$ concentrated on $\overline{z}$ and
$\underline{z}$, i.e. for any continuous bounded functional $F,$
defined on the space $\mathbb{C}[0, \infty)$, the next equality is valid:
$$\lim_{\varepsilon\to
0}\int_{\mathbb{C}[0,\infty)} F(f)\nu_\varepsilon(df)=
\int_{\mathbb{C}[0,\infty)} F(f)\nu(df).
$$
Moreover
\begin{equation}\label{osn2}
\lim_{\varepsilon\to 0}\int_{\mathbb{C}[0,\infty)}
F(f)\nu_\varepsilon(df)=\widetilde{\Gamma} F(\overline{z})+
\big(1-\widetilde{\Gamma}\big)F(\underline{z}),
\end{equation}

The limit measure $\nu$ is given by the right-hand side of equality
(\ref{osn2}); $\overline{z}, \underline{z}$ are extreme solutions of
the problem (\ref{zdr2}), and the value of $\widetilde{\Gamma}$ for
functions $\widetilde{b}(t)$, $\widetilde{\sigma}(t)$ is defined by
\cite[formula (2.11)]{krmah13}.

If we substitute in \cite[formula (2.11)]{krmah13} functions
$\widetilde{b}(t)$, $\widetilde{\sigma}(t)$, we get $\widetilde{\Gamma}=\Gamma$, where value of $\Gamma$ is defined
in (\ref{gamma}).

Further from the definition of a measure generated by the process we have:
$$\mu_\varepsilon\{A\}=\pr\{\xi_\varepsilon(\cdot)\in A\}=
\pr\{\kappa(\eta_\varepsilon(\cdot))\in A\}=
\pr\{\eta_\varepsilon(\cdot)\in
\varphi(A)\}=\nu_\varepsilon\{\varphi(A)\}.$$
From (\ref{osn2}) and
lemma \ref{lem2} we can get
$$\nu(\varphi(A))=\Gamma I_{\{\overline{z}(\cdot)\in \varphi(A)\}}+(1-\Gamma) I_{\{\underline{z}(\cdot)\in \varphi(A)\}}=$$
$$=\Gamma I_{\{\kappa(\overline{z}(\cdot))\in A\}}+(1-\Gamma) I_{\{\kappa(\underline{z}(\cdot))\in A\}}=\Gamma I_{\{\overline{y}(\cdot)\in A\}}+(1-\Gamma) I_{\{\underline{y}(\cdot)\in A\}}=\mu(A),$$
where the measure $\mu$ is defined by the right-hand side of equality (\ref{osn}).

So for any continuous bounded functional $F,$ defined on the space $\mathbb{C}[0, \infty)$, we have
$$\lim_{\varepsilon\to 0}\int_{\mathbb{C}[0,\infty)}F(y)\mu_\varepsilon\{dy\}=
\lim_{\varepsilon\to 0}\int_{\mathbb{C}[0,\infty)}F(y)\nu_\varepsilon\{\varphi(dy)\}=$$
$$=\lim_{\varepsilon\to 0}\int_{\mathbb{C}[0,\infty)}F(\kappa(y)) \nu_\varepsilon\{dy\}=\int_{\mathbb{C}[0,\infty)}F(\kappa(y))\nu\{dy\}=$$
$$=\int_{\mathbb{C}[0,\infty)}F(y)\nu\{\varphi(dy)\}=\int_{\mathbb{C}[0,\infty)}F(y)\mu\{dy\}.$$

{\it Theorem \ref{thbaldi} is proved.}
\vskip 5 pt

{\it Proof of theorem \ref{2}}. Let's consider the cases $A_2$ or $A_4$, cases $A_3$ or $A_5$ one can consider by analogy.  Similarly to the proof of theorem \ref{thbaldi} initially for the process $\eta_\varepsilon(t)$ by \cite[theorem 4.3]{krmah13} (with appropriate modifications) prove convergence of measures generated by the solutions of the equation (\ref{sdrito}) to the measure that is concentrated on the upper extreme solution of the corresponding Cauchy problem (\ref{zdr2}) and then return to the process $\xi_\varepsilon(t)$.

{\it Theorem \ref{2} is proved.}

\vskip 5 pt

\section{Examples.}

\begin{example} Let in the equation (\ref{sdrlch}) coefficients have the form
$$b(x)=\left\{\aligned
x^{\alpha_1}, \quad &  x\geq 0 , \\
 -C|x|^{\alpha_2},\quad & x\leq0 ,
\endaligned  \right.$$
$$\sigma(x)=\left\{\aligned
\sigma_1, \quad &  x\geq 0 , \\
 \sigma_2,\quad & x < 0 ,
\endaligned  \right.$$
with constants $C>0,$ $\sigma_i >0,$ $0<\alpha_i\leq1,$ $i=1,2.$

If $\alpha_1<1, \alpha_2=1,$ then the first integral in (\ref{cond1}) converges, and second one diverges, so we have the case $A_2$ and $\underline{y}(t)=0 $ by the lemma \ref{lem1}. Similarly, if $\alpha_1=1 , \alpha_2<1,$  then $\overline{y}(t)=0$.

If $0<\alpha_1, \alpha_2<1$, then we have the case $A_1$ and
$$L(x)=\left\{\aligned
\frac{x^{\alpha_1+1}}{\sigma_1^2(\alpha_1+1)}, \quad &  x\geq 0, \\
 -\frac{C|x|^{\alpha_2+1}}{\sigma_2^2(\alpha_2+1)}, \quad & x<0 .
\endaligned \right.$$

Then conditions (\ref{glem11}) and (\ref{glem21}) are hold with the constants
$\gamma=0,\; d=\frac{1}{\sigma_1^2(\alpha_1+1)},\;\delta=\alpha_1+1;$ $\;\; \theta=0, k=\frac{C}{\sigma_2^2(\alpha_2+1)},\; \mu=\alpha_2+1,$ which means that the conditions of theorem \ref{thbaldi} take place. We have: \\
1. If $\alpha_1=\alpha_2=\alpha<1$, then
$$ \Gamma=\frac{1}{1+\frac{1-\beta}{1+\beta}\bigg(\frac{C\sigma_1^2}{\sigma_2^2}\bigg)^{\frac{1}{\alpha+1}}}.$$\\
2. If $\alpha_1<\alpha_2\leq 1$, then $\Gamma=1$. \\
3. If $\alpha_2<\alpha_1\leq 1$, then $\Gamma=0$.

From the theorems \ref{thbaldi}, \ref{2} we have that the limit
measure is concentrated with weight $\Gamma$ on the upper extremal
solution and with weight $1-\Gamma$ on the lower extremal solution
of the corresponding Cauchy problem.
\end{example}

\begin{example} Let in the equation (\ref{sdrlch}) coefficient $\sigma(x)$ be such as in the example 1 and the drift coefficient be
$$b(x)=\left\{\aligned
x^\alpha\Big(|\ln x|+1\Big), \quad &  x> 0 , \\
 -|x|^\alpha,\quad & x\leq0 ,
\endaligned \quad 0<\alpha<1. \right.$$

Then the condition $A_1$ and the conditions of the theorem \ref{thbaldi} are hold. Value (\ref{gamma})
can be calculated by the lemma \ref{glem2}, because conditions (\ref{glem11}) and (\ref{glem21}) take place with constants $\gamma=-1, d=\frac{1}{\sigma_1^2(\alpha+1)},$ $\delta=\alpha+1;\;\; \theta=0, k=\frac{1}{\sigma_2^2(\alpha+1)}, \mu=\alpha+1.$

According to the lemma \ref{glem2}, we have $\Gamma=1$ and limit
measure is concentrated on the upper extremal solution of the Cauchy
problem (\ref{zdr}).
\end{example}

\begin{example} Let coefficients of the equation (\ref{sdrlch}) have such form:
$$b(x)=\left\{\aligned
x^\alpha, \quad &  x\geq 0 , \\
 -|x|^\alpha,\quad & x\leq0 ,
\endaligned \qquad 0<\alpha<1. \right.$$

$$\sigma(x)=\left\{\aligned
2-\cos x, \quad &  x\geq 0 , \\
2+\cos x,\quad & x<0 .
\endaligned  \right.$$

In this case, the condition $A_1$ and other conditions of the theorem \ref{thbaldi} are hold, thus there is a weak convergence.

According to the lemma \ref{glem2}
$ \displaystyle\Gamma=\Big(1+\frac{1-\beta}{1+\beta}9^{\frac{1}{\alpha+1}}\Big)^{-1}.$
So limit measure is concentrated with weight $\Gamma$ on the upper extreme solution and with weight $1-\Gamma$ on the lower extremal solution of the corresponding Cauchy problem.
 \end{example}

\end{document}